\documentclass[a4paper,11pt]{amsart}
\usepackage{graphicx,amssymb}
\usepackage[english]{babel}
\usepackage[applemac]{inputenc}

\newtheorem{thm}{Theorem}[section]

\newtheorem{lem}[thm]{Lemma}
\newtheorem{prop}[thm]{Proposition}
\newtheorem*{thm A}{Theorem A}
\newtheorem*{thm B}{Theorem B}
\theoremstyle{definition}
\newtheorem{defn}[thm]{Definition}

\newtheorem{problem}[thm]{Problem}

\theoremstyle{remark}
\newtheorem{rem}[thm]{Remark}
\newtheorem{conv}[thm]{Convention}

\newcounter{numl}

\newcommand{\labelnuml}{\textup{(\roman{numl})}}
\newenvironment{numlist}{\begin{list}{\labelnuml}%
{\usecounter{numl}\setlength{\leftmargin}{0pt}%
\setlength{\itemindent}{2\parindent}%
\setlength{\itemsep}{\smallskipamount}\def
\makelabel ##1{\hss \llap {\upshape ##1}}}}{\end{list}}

\def\cal{\mathcal}
\def\bb{\mathbb} 

\def\a{\alpha }

\def\d{\delta }
\def\D{\Delta}
\def\e{\epsilon }

\def\l{\lambda }

\def\o{\omega }

\def\s{\sigma }

\def\t{\theta }

\def\ot{\otimes }
\def\part{\partial }
\def\bpart{\bar\partial }

\def\w{\wedge }

\begin{document}


\title[Locally conformally symplectic structures] {Locally conformally symplectic  structures on compact  non-K\"ahler complex surfaces}

\author[V. Apostolov]{Vestislav Apostolov} \address{Vestislav Apostolov \\
D{\'e}partement de Math{\'e}matiques\\ UQAM\\ C.P. 8888 \\ Succ. Centre-ville
\\ Montr{\'e}al (Qu{\'e}bec) \\ H3C 3P8 \\ Canada}
\email{apostolov.vestislav@uqam.ca}

\author[G. Dloussky]{Georges Dloussky}\address{Georges Dloussky, Aix-Marseille University, CNRS, Centrale Marseille, I2M, UMR 7373, 13453\\
39 rue F. Joliot-Curie 13411\\
Marseille Cedex 13, France}
\email{georges.dloussky@univ-amu.fr}

\thanks{V.A. was supported in part by an NSERC discovery grant and is grateful to
the Universit\'e Aix-Marseille  and the Institute of Mathematics and Informatics of the Bulgarian Academy of
Sciences where a part of this project was realized. G.D. is grateful to the UQAM for the hospitality during the preparation of the work.  V.A. and G. D. were supported in part by the ANR grant  MNGNK-ANR-10-BLAN-0118. The authors are very grateful to Dimiter Vassilev for his invaluable help with the spectral theory of strongly elliptic operators, and to Dimitry Jakobson,  Fran\c{c}ois Hamel  and Fr\'ederic Rochon for sharing with us their expertise on spectral analysis and the theory of elliptic PDE's. They would also like to thank  Liviu Ornea and Misha Verbitsky to clarifying their results and giving an access to the unpublished  work \cite{OV4},  and to the referee for her/his valuable suggestions for improving the presentation.}

\date{\today}

\begin{abstract} 
We prove that every compact complex surface with odd first Betti number admits a locally conformally symplectic $2$-form which tames the underlying almost complex structure.
\end{abstract}

\maketitle


\section{Introduction}

It is a well-known result~\cite{siu, todorov, buchdahl, lamari} that a compact complex surface $S=(M,J)$ admits a K\"ahler metric if and only if its first Betti number $b_1(M)$ is even.   A cornerstone for the proof of this result is the fact, proved independently in \cite[Lemme~II.3]{gauduchon} and  \cite[p.~185]{HL} (see also \cite[p.~143, Prop.~1.6]{siu} for the case of a K3 surface),  that  $b_1(M)$ is even if and only if $M$ admits a symplectic form $\omega$ which {\it tames} $J$, in the sense that the $(1,1)$ part of $\omega$ is positive definite. This and the methods of proof in \cite{buchdahl,lamari} inspired the so-called ``{\it tamed to compatible}'' conjecture in symplectic geometry,  which asks whether an almost complex structure on $M$ which is tamed by a symplectic form admits a compatible symplectic form, see \cite{donaldson, taubes}.

\vspace{0.2cm} A natural extension of the theory of K\"ahler manifolds to the non-K\"ahlerian  complex case  can be obtained through the notion of  {\it locally conformally K\"ahler} metrics,  introduced and studied  in foundational work by F. Tricerri and I. Vaisman,  see e.g.~\cite{DO, OV} for an overview.  Recall that a {locally conformally K\"ahler} (or \emph{lcK}) metric on a complex manifold $X=(M,J)$ is defined by a positive-definite $(1,1)$-form $F$  satisfying  $dF= \theta \wedge F$ for a closed $1$-form $\theta$. The $1$-form $\theta$ is uniquely determined and is referred to as {\it the Lee form} of $F$. The corresponding Hermitian metric $g(\cdot, \cdot)= F(\cdot, J \cdot)$ defines a conformal class  $c$ on $M$. Changing the Hermitian metric $\tilde g = e^f g$ within $c$ amounts to transform the Lee form by $\tilde \theta = \theta + df$, so that the de Rham class $[\theta]$ is an invariant of the conformal class $c$.   

Of particular interest is the case of compact complex surfaces, where recent works~\cite{B,B1,B2,GO} showed that lcK metric exists for all known examples of compact complex surfaces with odd first Betti number,  with the exception of the complex surfaces obtained by blowing up points of certain Inoue surfaces with zero second Betti number,  described in \cite{B}.  However, a general existence result is still to come.

\vspace{0.5cm}
In this paper we study, on a compact complex surface  $S=(M,J)$ with odd first Betti number,  the problem of existence of {\it locally conformally symplectic forms} $\omega$ which tame $J$, i.e.   $2$-forms $\omega$ satisfying $d\omega = \theta \wedge \omega$ for a closed $1$-form $\theta$ (called Lee form of $\omega$),  and such that the $(1,1)$-part of $\omega$ is positive-definite. This  is, in general,  a weaker condition than the existence of lcK metrics, which turns out to be related to the theory of {\it bihermitian} conformal structures~\cite{AGG,Pontecorvo} in the case  when the $\omega$-conjugate  of $J$ determines another   integrable almost-complex structure $J^{\omega}$ on $M$, see \cite{ABD}. 

We  establish  the following general existence result,  which we believe is an important step towards the resolution of  the existence problems for both lcK and bihermitian conformal structures on a non-K\"ahler complex surface, and 
which  answers in positive (in the case of complex surfaces)  a question raised in \cite[Open Problem 1]{OV}.
\begin{thm}\label{thm:lcs}  Any compact complex surface $S=(M,J)$ with odd first Betti number admits a locally conformally symplectic form $\omega$ which tames $J$.
\end{thm}
The above theorem is derived from another existence result concerning a conformal class of Hermitian metrics on $S=(M,J)$, which can be regarded as a twisted version of  Gauduchon's celebrated theorem~\cite{gauduchon-0},  and can be stated as follows. Let $a \in H^1_{dR}(M)$ be a de Rham cohomology class and $\alpha \in a$ a closed $1$-form in $a$. Denote by $d_{\alpha}:=d-\alpha\wedge.$ the twisted differential operator  defining  the {\it Lichnerowicz--Novikov} complex,  and let $d_{\alpha}^c: = J d J^{-1}$. 
\begin{thm}\label{thm:conformal} Let $S=(M,J)$ be a compact complex surface with odd first Betti number,  and $c$ a conformal class of Hermitian metrics on $S$. Then, there exists a non-zero de Rham class $a \in H^1_{dR}(M)$ such that for any metric $g \in c$, there exists a representative $\alpha \in a$ such that the fundamental $2$-form $F$ of $g$ satisfies
$$d_{\alpha} d^{c}_{\alpha} F =0.$$
\end{thm}
The de Rham class $a$ appearing in Theorem~\ref{thm:conformal}  determines, via the exponential map 
\begin{equation}\label{H1}
H^1_{dR}(S,\bb C)\stackrel{\exp}{\hookrightarrow}H^1(S,\bb C^*) \longrightarrow {\rm Pic}_0(S), 
\end{equation}
a flat holomorphic line bundle $\mathcal L_a$, and we derive Theorem~\ref{thm:lcs} from Theorem~\ref{thm:conformal} by showing that,  when $H^2(S, \mathcal L_{a})=\{0\}$, $S$ also admits  a locally conformally symplectic form  with a Lee form $\alpha$. The proofs of Theorems~\ref{thm:lcs} and Theorem~\ref{thm:conformal} are presented in Section~\ref{s:lichnerowicz-novikov} of the paper, whereas the necessary analytical tools are collected in the Appendix~\ref{a:PDE}.

\vspace{0.5cm}
In the light of the ``tamed to compatible'' conjecture  mentioned above, it is natural to compare the existence on $S$ of lcK metrics and of locally conformally symplectic forms taming $J$,  through the corresponding de Rham classes of their Lee forms. We thus introduce in Section~\ref{s:cohomological} the subset $\mathcal{C}(S)$ (resp. $\mathcal{T}(S)$) in  $H^1_{dR}(M)$ of  classes $a$ for which there exists a lcK metric on $S$ with Lee form $\theta \in a$ (resp. for which there exists a locally conformally symplectic form which tames $J$, with Lee form in $a$). We obviously have the inclusion $\mathcal{C}(S) \subseteq \mathcal{T}(S)$, and one may ask (see also~\cite{OV}, \cite[Rem.~9]{B1}):

\begin{problem}\label{problem:secondary} Let $S=(M,J)$ be a compact complex surface with odd first Betti number.  Determine the set $\mathcal{T}(S)\subset H^1_{dR}(M)$ of classes $a$ for which there exists a  locally conformally symplectic form $\omega$ which tames $J$ and has a Lee form $\theta \in a$. Is $\mathcal{T}(S)$ strictly bigger than $\mathcal{C}(S)$?
\end{problem}
Our initial motivation to study  the above problem came from the theory of bihermitian conformal structures  developed  in \cite{ABD}, where the existence of the latter was reduced to answering the question of whether certain  classes $a \in H^1_{dR}(M)$ belong to $\mathcal{T}(S)$ and $\mathcal{C}(S)$.

A number of partial results concerning  Problem~\ref{problem:secondary} are obtained in Section~\ref{s:surfaces}, where we specialise to  the case of a compact complex surface with first Betti number equal to $1$. 
 
In the last Section~\ref{s:examples} of the paper,  we consider  some examples of non-K\"ahler complex surfaces in the Kodaira class VII (i.e. satisfying $H^0(S, \mathcal{K}_S^{\ell})=\{0\}$ for all $\ell \ge 1$, where $\mathcal{K}_S$ stands for the canonical line bundle of $S$,  see \cite{bpv}),  for which  a complete answer to the above problem can be given. It is known that in this case,  the first Betti number equals to $1$ (see e.g. \cite{bpv}),  and that the degree with respect to some Gauduchon metric on $S$ of the holomorphic line bundles determined via \eqref{H1} induces an orientation on $H^1_{dR}(S) \cong (\mathbb{R}, >)$, which turns out to be independent of  the choice of a Gauduchon metric (see \cite[Rem.~2.4]{Teleman} or Lemma~\ref{l:sign}).  Thus, for any compact complex surface in the Kodaira class VII, one can naturally identify $H^1_{dR}(S)$ with the oriented real line $(-\infty, + \infty)$. In this notation,  a combination of Propositions~\ref{p:hopf} and \ref{p:inoue} gives
 \begin{thm}\label{thm:BH} Let  $S$ be a compact complex surface whith a minimal model $S_0$. \begin{enumerate}
 \item[(i)] If $S_0$ is  a Hopf surface, then $\mathcal{T}(S)=\mathcal{C}(S)=(-\infty, 0)$.  
 \item[(ii)] If $S$ is an Inoue surface of the type $S^+_{N,p,q,r; u}$ with $u\in \mathbb{C}\setminus \mathbb{R}$, then $\mathcal{C}(S) =\emptyset$ and   $\mathcal{T}(S)= \{a_0\}.$
 \item[(iii)] If $S$ is an Inoue surface of the type $S^+_{N,p,q,r; u}$ with $u\in  \mathbb{R}$, then $\mathcal{C}(S) =\mathcal{T}(S)= \{a_0\},$
 \end{enumerate}
 where $a_0\in H^1_{dR}(S)$ denotes the de Rham class for which the holomorphic line bundle determined by \eqref{H1} is isomorphic to the anti-canonical line bundle $\mathcal{K}^*_S$. 
 \end{thm}

\section{Existence of locally conformally symplectic forms taming the complex structure}\label{s:lichnerowicz-novikov}
Throughout the paper, we shall use the following
\begin{conv}\label{convention}
{\rm Let $\alpha$ be a closed $1$-form  on $M$, representing  a de Rham class $a=[\alpha]$. We denote by $L_{\alpha}$ the trivial real line bundle 
over $M$,  endowed with a (non-trivial) flat connection $\nabla^{\alpha} s := d s + \alpha\otimes s$,  where $s$ is a smooth section of $L$. Similarly, $\nabla^{\alpha}$ induces a holomorphic structure on the complex bundle  $\mathcal{L}_{\a} := L_{\alpha} \otimes \mathbb{C}$. Writing $\a_{\mid U_i} = df_i$ on an open covering $\mathfrak{U}= (U_i)$ of $M$,  $\{(U_i, e^{-f_i})\}$ defines a parallel (respectively holomorphic) trivialization of $L_{\alpha}$ (resp. $\cal L_{\a}$) with transition functions $e^{f_i-f_j}$ on $U_i \cap U_j$. (With respect to this trivialization, $s_0=(U_i, e^{f_i})$ is a nowhere vanishing smooth section of $L_{\alpha}$.) In terms of \eqref{H1}, $L_{\alpha}$ (resp. $\cal L_{\a}$) is isomorphic to the flat real line bundle $L_{a}$ (resp. the flat holomorphic line bundle $\cal L_{a}$) determined by the deRham class $a=[\alpha]$.  In what follows, we shall tacitly identify $L_{\a}$ and $\mathcal L_{\a}$ with the bundles $L_{a}$ and $\cal L_{a}$, respectively,  and simply refere to them as the  flat bundles {\it corresponding} to $a\in H^1_{dR}(M)$. }
\end{conv}
Let $L=L_{a}$ be the flat real line bundle over $M$ determined by $\a \in a, \ a\in H^1_{dR}(M)$ via  the Convention~\ref{convention},  and denote by $L^*= L_{-a}$ its dual. 
The differential operator  $d_{\alpha}=d-\alpha\wedge.$ defines the Lichnerowicz--Novikov complex
\begin{equation}\label{lichne-novikov}
\cdots \stackrel{d_{\alpha}}{\to}\Omega^{k-1}(M) \stackrel{d_{\alpha}}{\to} \Omega^{k}(M)  \stackrel{d_{\alpha}}{\to}\cdots 
\end{equation}
which is  isomorphic to the de Rham complex of differential forms with values in $L^*$
\begin{equation}\label{twisted-deRham}
\cdots \stackrel{d_{L^*}}{\to}\Omega^{k-1}(M,L^*) \stackrel{d_{L^*}}{\to} \Omega^{k}(M, L^*)  \stackrel{d_{L^*}}{\to}\cdots.
\end{equation}
In particular,  we have an isomorphism between the cohomology groups
$$H^{k}_\a(M)\simeq H^k_{d_{L^*}}(M,L^*).$$
Considering the Dolbeault cohomology groups of $S$ with values in the flat holomorphic line bundle ${\mathcal L}^*$, we have 
$$d_{L^*}=\part _{\mathcal {L^*}}+\bpart_{\mathcal {L}^*}, \quad {\rm and}\quad d_{\alpha}=\part_{\alpha}+\bpart_{\alpha}$$
 with 
 $$\part_{\alpha}=\part-\alpha^{1,0}\w\quad {\rm and}\quad \bpart_{\alpha}=\bpart-\alpha^{0,1}\w,$$
giving rise to the isomorphisms 
\begin{equation}\label{dolbeault}
H^{p,q}_{\bar \partial_{\alpha}}(S)\simeq H_{\bar\partial_{\cal L^*}}^{p,q}(S,\cal L^*).
\end{equation}

\begin{defn} Let $X=(M,J)$ be a complex manifold. We shall say that a differentiable 2-form $\o$ is a  locally conformally  symplectic form taming $J$ if there exists  a closed differentiable 1-form $\alpha$ such that  $d_{\alpha}\o=0$, and the $(1,1)$-part $\o^{1,1}$ of $\omega$ is positive definite. If, furthermore, $\o$ is of type $(1,1)$, it defines a locally conformally K\"ahler structure on $X$. The $1$-form $\alpha$ is called  the Lee form $\alpha$ of $\omega$.
\end{defn}
\begin{rem}\label{r:conformal}
In terms of the isomorphism between \eqref{lichne-novikov} and \eqref{twisted-deRham}, if we write $\alpha_{\mid U_i}=df_i$ on an open covering  $\mathfrak {U}=(U_i)$ of $M$, then ${\omega_i}_{\mid U_i} := e^{-f_i}\omega$ defines a $d_{L^*}$-closed $2$-form with values in $L^*$,  whose $(1,1)$-part is positive definite (for the latter we use the fact $L^*= L_{-a}$ is defined by the co-cycle  $(e^{f_j-f_i}, U_{ij} = U_i\cap U_j)$  which consist of positive constant real functions). Similarly, if $\omega$ is a  locally conformally  symplectic form taming $J$ with $d_{\alpha}\omega=0$, then $\tilde \omega = e^{f} \omega$ is a locally conformally symplectic  form taming $J$ which satisfies $d_{\tilde \alpha} \tilde \omega =0$ with $\tilde \alpha = \alpha + df$. It follows that the existence of a locally conformally symplectic form taming $J$  with Lee form $\alpha$ merely depend upon the de Rham class $a=[\alpha]$.
\end{rem}
Our first observation is the following
\begin{lem}\label{from symplectic to SKT} Let $\omega$ be a  locally conformally  symplectic form on $X=(M,J)$ with  Lee form $\alpha$. Denote by $F:= \omega^{1,1}$ the $(1,1)$-part of $\omega$. Then,
$$d_{\alpha} d^c_{\alpha} F =0,$$
where $d^c_{\alpha}= J d_{\alpha} J^{-1}=i(\bar\partial_{\alpha} - \partial_{\alpha})$.
\end{lem}
\begin{proof} Writing 
$$\o=F +\o^{2,0}+\o^{0,2}$$
where $\o^{2,0}$ and $\o^{0,2}$ denote the  $(2,0)$ and $(0,2)$-part of $\omega$,  respectively. As $d_{\alpha} = \partial_{\alpha} + \bar\partial_{\alpha}$,  one  has
$$d_\alpha \o=0 \Longleftrightarrow \Big\{ \begin{array}{lcl} \part_{\alpha} F +\bpart_{\alpha} \o^{2,0} &=& 0 \\ 
\bpart_{\alpha} F+\part_{\alpha}\o^{0,2} &=&0, \end{array} $$
hence $2i\part_{\alpha}\bpart_{\alpha} F =d_{\alpha}d^c_{\alpha} F=0$.
\end{proof}
\begin{lem}\label{l:twisted-gauduchon} Let $S=(M,J)$ be a complex surface and $g$ a Hermitian metric with fundamental $2$-form $F(\cdot, \cdot)=g(J\cdot, \cdot)$.  Then,
$$d_{\alpha} d^c_{\alpha} F=0 \quad \Longleftrightarrow \quad \d(\t-\alpha)+g(\t-\alpha,\alpha)=0,$$
where $\delta$ is the co-differential with respect to $g$ and $\theta = J \delta F$ is the Lee form of $g$.
\end{lem}
\begin{proof} Using $dF = \theta \wedge F$ (which, as $F$ is a self-dual $2$-form with respect to $g$ is equivalent to the relation $\theta = J \delta F$), one gets
$$d_{\alpha}d_{\alpha}^cF=\Bigl[d\bigl(J(\t-\alpha)\bigr)+(\t-\alpha)\w J(\t-\alpha)\Bigr]\w F.$$
As $F$ is self-dual, i.e. $* F = F$ where $*$ stands for the Hodge operator with respect to $g$, $d_{\alpha}d_{\alpha}^cF=0$   is equivalent to
\begin{equation*}
\begin{split}
0 &=g\Bigl(d\bigl(J(\t-\alpha)\bigr)+(\t-\alpha)\w J(\t-\alpha), F\Bigr)\\
   & = \sum_{i=1}^4 \Big(D_{e_i} \big(J(\t-\alpha)\big)(Je_i)\Big) +|\t -\alpha|^2_g \\
   &= \sum_{i=1}^4 \Big((D_{e_i}J)(\theta - \a) (Je_i) + \big(D_{e_i}(\theta-\a)\big)(e_i)\Big) + |\t -\alpha|^2_g \\
   &= (\theta-\a)\Big(J \sum_{i=1}^4 (D_{e_i} J)(e_i)\Big) -\d(\t-\alpha)+ |\t-\alpha|^2_g \\
    &= - g(\theta, \t-\alpha) -\d(\t-\alpha)+  g(\t-\alpha,\t-\alpha)\\
    & =-\d(\t-\alpha) -g(\t-\alpha,\alpha),
\end{split}
\end{equation*}   
where $D$ is the riemannian connection of $g$,  $\{e_i, \ i=1, \cdots, 4\}$ is any $J$-adapted orthonormal frame, and for passing from the 3th line to the 4th we have used the fact that $D_{e_i}J$ is skew-symmetric with respect to $g$ and 
anti-commutes with $J$ whereas for passing from the  4th line to the 5th  we have used  the identity $\theta(X)= (\delta F)(-JX)= -\sum_{i=1}^{4}g(J(D_{e_i}J)(e_i),X)$.                                                                                                       
\end{proof}
\begin{rem}\label{r:twisted-gauduchon} In the light of Remark~\ref{r:conformal}, it is easily seen that the condition
\begin{equation}\label{twisted-gauduchon}
\d(\t-\a)+g(\t-\a ,\a)=0
\end{equation}
is also conformally invariant. More precisely, it is straightforward to check that  if $\tilde g=e^fg$ and $\tilde \a=\a+df$, then \eqref{twisted-gauduchon} is satisfied for $(g,\a)$ if and only if it is satisfied for $(\tilde g,\tilde \a)$.
\end{rem}

Recall the fundamental result of Gauduchon~\cite{gauduchon-0} which affirms that if $X=(M,J)$ is an $m$-dimensional compact complex manifold endowed with a Hermitian metric $g$, then there exists (a unique up to scale) conformal metric $\tilde g = e^f g$ whose fundamental form $\tilde F$ satisfies $dd^c \tilde F^{m-1}=0$, or equivalently, for which $\tilde\delta J \tilde \delta \tilde F =0$. Such Hermitian metric is referred to as {\it Gauduchon metric}. By Lemma~\ref{l:twisted-gauduchon} and Remark~\ref{r:twisted-gauduchon}, we then obtain

\begin{prop}\label{p:twisted-gauduchon} Let $S=(M,J)$ be a compact complex surface,  $a\in H^1_{dR}(M)$ a de Rham class and $c=[g]$ a conformal class of Hermitian metrics on $M$, with $g$ a Gauduchon metric in $c$. Then the following conditions are equivalent:
\begin{enumerate}
\item[\rm(i)] for any $\tilde g \in c$, the fundamental $2$-form $\tilde F$ satisfies $d_{\tilde \a} d^c_{\tilde \a} \tilde F=0$ for some $\tilde \alpha \in a$;
\item[\rm(ii)] there exists a positive smooth function $\psi$ on $M$ which satisfies the equation
\begin{equation}\label{L}
\mathbb{L}_{g,a} (\psi)= \Delta_g (\psi) - g(\theta^g - 2a_h^g, d\psi) + g(\theta^g-a_h^g, a_h^g) \psi =0,
\end{equation}
where  $\Delta_g$ is the Riemannian Laplace operator of $g$, $\theta^g$ is the corresponding co-closed Lee form,  and $a_h^g$ is the harmonic representative for $a$ with respect to $g$.
\end{enumerate}
\end{prop}
\begin{proof} By Remark~\ref{r:twisted-gauduchon}, if for some metric $\tilde g \in c$ the corresponding fundamental form $\tilde F$ satisfies $d_{\tilde \a}d^c_{\tilde \a} \tilde F =0$, for any other metric $g = e^{-f} \tilde g$ in $c$, the fundamental $2$-form $F$ satisfies $d_{\a}d^c_{\a} F=0$ with $\a = \tilde \a - df$. It  follows that the condition (i) is equivalent to the existence of a $1$-form $\a \in a$ such that the fundamental $2$-form of a Gauduchon metric $g\in c$ satisfies $d_{\a}d^c_{\a} F=0$. Writing $\a = a_h^g - d \log \psi$ with $\psi>0$,  Lemma~\ref{l:twisted-gauduchon} reads
\begin{equation}\label{computing}
\begin{split}
                              0 &= \d(\t^g-\a)+g(\t^g-\a ,\a) = - \d\a + g(\t^g-\a, \a) \\
                                 &= - \d (a_h^g - d\log \psi) + g(\t^g-a_h^g + d\log \psi, a_h^g - d\log \psi) \\
                                 &= \Delta_g (\log \psi)  -g(\t^g - 2a_h^g, d\log \psi)  + g(\t^g - a_h^g, a_h^g) - g(d \log \psi, d\log \psi) \\
                                 &= \frac{1}{\psi} \Big(\Delta_g \psi - g(\t^g - 2a_h^g, d\psi)  + g(\t^g-a_h^g, a_h^g)\psi\Big),
\end{split}
\end{equation}
and the claim follows. \end{proof}
The  following elementary observation regarding the linear operator $\mathbb{L}_{g,a}$ will be used throughout.
\begin{lem}\label{l:integration-by-parts} For any  everywhere positive smooth function $\psi$ on $M$ 
\begin{equation}\label{eq}
\int_M \frac{\mathbb{L}_{g,a}(\psi)}{\psi} v_g = \int_M g(\theta^g, a_h^g) - \int_M \big(||a_h^g||_g^2 + \frac{1}{\psi^2} ||d\psi||^2_g\big)v_g.
\end{equation}
In particular, if \eqref{L} admits a positive solution with respect to some $a \neq 0 \in H^1_{dR}(M)$, then
\begin{equation}\label{ineq}
\int_M g(\theta^g, a_h^g) v_g = \int_M g(\theta^g_h, a_h^g) v_g = \int_M \big(||a_h^g||_g^2 + \frac{1}{\psi^2} ||d\psi||^2_g\big)v_g>0.
\end{equation}
\end{lem}
\begin{proof} The equality \eqref{eq} follows by integration by parts and using that $g$ is Gauduchon, i.e. $\delta \theta^g =0$; the inequality \eqref{ineq} is an immediate consequence. \end{proof}

The general theory for the existence of positive solutions of the elliptic linear second-order PDE $\mathbb{L}_{g,a} (\psi)=0$ is reviewed in the appendix~\ref{a:PDE} to this paper. We recollect below the following variational characterization.

\begin{prop}\label{p:variational} The PDE \eqref{L} has a positive solution $\psi$ if and only if 
\begin{equation}\label{variational}
\lambda_a(g) := \sup_{u \in \mathcal{C}^{\infty}(M),  u>0} \Big\{ \min_{M} \frac{\mathbb{L}_{g,a}(u)}{u}\Big\}=0.
\end{equation}
Furthermore, $\lambda_{a}(g)$ is a finite number which varies analytically with respect to linear variations $g_t = (1-t)g  + t\tilde g$ of Gauduchon metrics, or linear variations $a_t = ta$ of de Rham classes.
\end{prop}

\vspace{0.2cm}
\noindent
{\bf Proof of Theorem~\ref{thm:conformal}}.  By Propositions~\ref{p:twisted-gauduchon} and \ref{p:variational}, it is enough to fix a Gauduchon metric $g \in c$ and show that $\lambda_a(g)=0$ for a suitable choice of $a \neq 0$.  Let $\theta_h^g$ denote the harmonic part of the Lee form of $g$. It is well-known (see e.g. \cite[Prop.~1]{AG}  or Proposition~\ref{non-kahler} below) that $[\theta_h^g]\neq 0$ under the assumption that the first Betti number is odd. We let
$a_t=\frac{t}{2}[\theta_h^g]$ with $t>0$ and denote by $a_h^t= \frac{t}{2}\theta_h^g$ the harmonic representative of $a_t$ with respect to $g$. We are going to show that $\lambda_{a_t}(g) <0$ when $t \ge 2$ and $\lambda_{a_t}(g) >0$ for $t$ close to $0$; by the second part of Proposition~\ref{p:variational}, this would imply that $\lambda_{a_t}(g) =0$ for some $t\in (0,2)$.

Let $u_t>0$ be the eigenfunction of $\mathbb{L}_{g,a_t}$, corresponding the the principal eigenvalue $\lambda(t):=\lambda_{a_t}(g)$, normalized by
$\int_M u_t^2 v_g =1$ (see Theorem~\ref{A:1} in the Appendix~\ref{a:PDE}). By Lemma~\ref{l:integration-by-parts} we have
\begin{equation}\label{lambda(t)}
\begin{split}
\lambda(t)  {\rm vol}_{g}(M) &= \int_M \frac{\mathbb{L}_{g, a_t}(u_t)}{u_t} v_{g}\\
                                                                                     & = -\int_M \frac{1}{u_t^2}||du_t||^2_{g} v_{g} + \int_M g ({\theta}_h^{g} - a_h^{t}, a_h^{t})v_{g}\\
                                                                                     & = -\int_M \frac{1}{u_t^2}||du_t||^2_{g} v_{g} + \frac{t(2-t)}{4} \int_M ||\theta_h^g||_g^2 v_{g}.
\end{split}
\end{equation}
Taking $t\ge 2$ in \eqref{lambda(t)} yields $\lambda(t) \le 0$. We are going to show that $\lambda(t)>0$ for positive $t$ close to $0$.   Indeed, as for each $t$ $\lambda(t)$ is a simple eigenvalue (see Theorem~\ref{A:1}),  by the Kato--Rellich theory~\cite{kato, rellich} (see Theorem~\ref{A:2} in the Appendix~\ref{a:PDE}), $\lambda(t)$  and $u_t$ vary analytically with respect to $t$. As ${\mathbb L}_{g,0}$ does not have a zero order term, by the Hopf maximum  principal (see e.g. ~\cite[III, Sect.~8, 3.71]{aubin})  $\lambda(0)=0$  and $u_0= 1/{\rm vol}_g(M)$.  It thus follows that the function $t \mapsto \int_M\frac{\|du_t||^2_g}{u_t^2} v_g$ has a global minimum $0$ at $t=0$, so differentiating \eqref{lambda(t)} at $t=0$ we get
$$\lambda'(0){\rm vol}_g(M) =\frac{1}{2} \int_M ||\theta_h^g||_g^2 v_g>0.$$
\hfill $\Box$

In order to obtain a converse of Lemma~\ref{from symplectic to SKT}, we first notice the following

\begin{lem}\label{vanishing} Let $\cal L$ be a  flat holomorphic line bundle  over a compact complex surface  $S$,  such that 
$H^0(S,\cal L)=H^2(S,\cal L)=\{0\}.$ Then,
$H^{2,1}_{\bpart_{\cal L^{*}}}(S,\cal L^{*})=\{0\}.$
\end{lem}
\begin{proof} As $H^{2,1}_{\bpart_{\cal L^*}}(S,\cal L^*) \cong H^1(S, \mathcal{K}_{S} \otimes \cal L^*)$,  by Serre duality
$${\rm dim} \  H^{2,1}_{\bpart_{\cal L^*}}(S,\cal L^*)= h^1(S, \mathcal{K}_S\ot \cal L^*)=h^1(\cal L).$$ 
Using the vanishing of the  cohomology groups from the hypothesis of the Lemma and the Riemann--Roch formula, one concludes
${\rm dim} \  H^{2,1}_{\bpart_{\cal L^*}}(S,\cal L^*)= 0.$  \end{proof}

Let $\mathcal{L}_{a}$ be the flat holomorphic line bundle corresponding to $a=[\alpha]$, see Convention~\ref{convention}. Its degree with respect to the Gauduchon metric $g$ is defined to be~\cite{gauduchon}
$${\rm deg}_g (\cal L_a) =  \frac{1}{2\pi}\int_M  \rho_{\cal L_a} \wedge F, $$
where $\frac{1}{2\pi}\rho_{\cal L_a}$ is a pluriharmonic representative of the first Chern class of $\cal L_a$. Writing $\a_{\mid U_i} = df_i$ on an open covering $\mathfrak{U}= (U_i)$ of $M$,  $\cal L_a$ is the topologically trivial complex line bundle over $M$ with $s_0=(U_i, e^{f_i})$ being a nowhere vanishing smooth section, whereas $h_{\mid U_i}= e^{-2f_i} (\cdot, \cdot)$ introduces a hermitian metric on $\cal L_a$ with  $h(s_0,s_0)=1$.  It follows that $\rho = -\frac{1}{2}dd^c \log e^{-2f_i} = -d^c \alpha$ is the curvature of the Chern connection of $(\cal L_a,h)$, so that
\begin{equation}\label{e:degree}
\begin{split}
{\rm deg}_g (\cal L_a)  &= -\frac{1}{2\pi} \int_M d^c \a \wedge F = -\frac{1}{2\pi} \int_M g(d^c\a, F) v_g \\
                                      &= - \frac{1}{2\pi} \int_M g(\theta^g, \a) v_g = - \frac{1}{2\pi} \int_M g(\theta^g_h, a_h^g)v_g,
\end{split}
\end{equation}                                      
where for the last line we have used that $d\a =0$, $\delta \theta^g =0$.  
It then follows from Proposition~\ref{p:twisted-gauduchon}, the inequality \eqref{ineq} of Lemma~\ref{l:integration-by-parts} and  the properties of the degree (see e.g. \cite{gauduchon})
\begin{lem}\label{vanish} Let $g$ be a Gauduchon Hermitian metric  on  a compact complex surface $S$,  whose fundamental $2$-form $F$ satisfies $d_{\a}d^c_{\a} F=0$ for some closed but not exact $1$-form $\alpha$. Then,  the flat holomorphic line bundle $\cal L_a$ determined by $a=[\alpha] \in H^1_{dR}(M)$ via \eqref{H1} and Convention~\ref{convention} satisfies 
$${\rm deg}_g(\cal L_a)= -\frac{1}{2}\int_M \Big(||a_h^g||_g^2 + \frac{1}{\psi^2} ||d\psi||^2_g\Big)v_g<0,$$
so that $H^0(S, \cal L_a^{\ell})=\{0\}$ for all $\ell \ge 1$.
\end{lem}
We thus obtain the following 

\begin{prop}\label{from SKT to symplectic} Let $g$ be a Gauduchon Hermitian metric  on  a compact complex surface $S$,  whose fundamental $2$-form $F$ satisfies $d_{\a}d^c_{\a} F=0$ for some closed but not exact $1$-form $\alpha$.  If the holomorphic flat bundle $\mathcal{L}_a$ corresponding to the deRham class $a=[\alpha]$ via \eqref{H1} and Convention~\ref{convention} satisfies $H^2(S, \cal L_a)=\{0\}$, then there exists a locally conformally symplectic $2$-form $\omega$ on $S$ with Lee form $\a$ and whose $(1,1)$-part  is $F$.
\end{prop}
\begin{proof}  As  $H^0(S, \cal L_a)=\{0\}$ by Lemma~\ref{vanish}, the hypothesis $H^2(S, \cal L_a)=\{0\}$, Lemma~\ref{vanishing}   and the isomorphism \eqref{dolbeault} imply $H^{2,1}_{{\bar \partial}_{ \cal L_a^*}} (S, \cal L_a^*) \cong H^{2,1}_{\bar \partial_{\a}}(S)=\{0\}$.   As $\bar\partial_{\a} \partial_{\a} F =  \frac{i}{2} d_{\a}d^c_{\a} F =0$, it  follows that 
$$\partial_{\a} F + \bar \partial_{\a} \beta=0, $$
for a $(2,0)$-form $\beta$.  Letting
$$\omega = F + \beta + \bar \beta, $$
one has
\begin{equation*}
\begin{split}
d_{\a} \omega &= (\partial_{\a} + \bar \partial_{\a}) (F + \beta + \bar \beta) \\
                           &= \partial_{\a} F + \bar \partial_{\a} \beta + \bar \partial_{\a} F  + \partial_{\a} \bar \beta =0,
\end{split}
\end{equation*}                           
which completes the proof. \end{proof}

\vspace{0.2cm}
\noindent
{\bf Proof of Theorem~\ref{thm:lcs}.}  Let $S=(M,J)$ be a compact complex surface with odd first Betti number $b_1(M)$. If the Kodaira dimension of $S$ is non-negative, it follows from \cite{B} and \cite{tricerri} that $S$ admits a locally conformally K\"ahler metric. We therefore suppose  that the Kodaira dimension of $S$ is negative, i.e. $S$ belongs to class VII of the Kodaira list, see e.g. \cite{bpv}.   Denote by $S_0$ the minimal model of $S$.

We first suppose that the second Betti number of $S_0$ is greater or equal to $1$, i.e.  $S_0$ is in the Kodaira class ${\rm VII}_0^+$ (for which a complete classification is still to come). In this case, Theorem~\ref{thm:lcs} follows from Theorem~\ref{thm:conformal}, Proposition~\ref{from SKT to symplectic} and the following vanishing result.

\begin{lem}\label{l:vanishing2} Let $S$ be a compact complex surface  with $b_1(S)=1$ and negative Kodaira dimension. Suppose that the second Betti number of its minimal model $S_0$ is $>0$. Then,  for any topologically trivial line bundle $\cal L\in {\rm Pic}_0(S)$, we have
$$H^2(S,\cal L)=0.$$
\end{lem}
\begin{proof} By Serre duality,  we have to show that for any $\cal L\in {\rm Pic}_0(S)$, $H^0(S,\mathcal{K}_S\ot \cal L)=0$. Suppose that there exists a non trivial  section $\s \in H^0(S,\mathcal{K}_S\ot \cal L)$. The line bundle 
$\mathcal{K}_S\ot \cal L$ is not trivial as otherwise  $\mathcal{K}^{-1}_S$ must be flat, and therefore  $0=c_1^2(S)=-b_2(S)=-c_2(S)$ (see \cite{bpv}), a contradiction.  It follows that $\s$ must vanish along an effective divisor $D$ with $[D]=\mathcal{K}_S\ot \cal L$. Therefore in $H^2(S,\bb Z)$,
\begin{equation}\label{adjunction}
0=(\mathcal{K}_S+\cal L-D)\cdot \mathcal{K}_S=\mathcal{K}_S^2-\mathcal{K}_S \cdot D.
\end{equation}
We show the assertion by induction on the number $p$ of blowing ups.

If $p=0$, the surface is minimal and by \cite[p.~399]{Nakamura}, an irreducible curve $C$ of $D$ is either
\begin{itemize}
\item a rational curve  such that $C^2\le -2$, therefore by the adjunction formula, $0=\pi(C)=1+\frac{\mathcal{K}_S \cdot C+C^2}{2}$ and $K_S\cdot C\ge 0$, or
\item an elliptic  or a rational curve $C$ with a double point. By the same formula, $\mathcal{K}_S\cdot C+C^2=0$, hence $K_S \cdot C\ge 0$,
\end{itemize}
We deduce that for any effective divisor $D$, $\mathcal{K}_S \cdot D\ge 0$. Moreover, as $b_2(S)=-c_1(S)^2=-\mathcal{K}_S^2>0$, we have a contradiction in \eqref{adjunction}.

If $p\ge 1$, there is an exceptional curve of the first kind $E$ and let $B_x : S\to \check{S}$ be the  blowing down of $E$ to a point $x\in \check{S}$. Let $\check{U}$ be a ball centred at $x$ and $U={B_x}^{-1}(\check{U})$ a simply connected neighbourhood of $E$. The line bundle $\cal L$ is holomorphically trivial on $U$ and, therefore,  the coherent sheaf $\check{\cal L}={(B_x)}_{*} \cal L$ is in fact a topologically trivial line bundle on $\check{S}$. As $(B_x)_* \mathcal{K}_S = \mathcal{K}_{\check S}$, any non-trivial section $\s\in H^0(S, \mathcal{K}_S\otimes \cal L)$ gives, via the biholomorphism  $B_x : S\setminus E \cong \check{S}\setminus\{x\}$ a non-trivial section $\check{\s}\in H^0(\check{S}\setminus\{x\},\check{\mathcal{K}}_{\check{S}}\otimes \check{\cal L})$. By Hartogs' theorem,  this section extends to $\check{S}$. As $b_2(\check{S}) \ge b_2(S_0)>0$,  by the induction hypothesis $\check{\s}=0$, a contradiction. \end{proof}

Let us now consider the case when $b_2(S_0)=0$.  According to \cite{bogomolov,andrei, yau-et-al}, $S_0$ is either an Inoue surface  (see \cite{inoue}) or a Hopf surface (see \cite{kato-hopf}). The arguments in \cite{tricerri,Vul} can be used without any change to show that the blow-up  $S$ of  $S_0$ admits a locally conformally symplectic form taming $J$ if $S_0$ does. We thus consider the case $S=S_0$.

The  Inoue surfaces with second Betti number equal to zero are classified by Inoue \cite{inoue} who  shows that they do not admit  curves. Thus, if $S$ is such a surface, the condition $H^2(S, \mathcal{L}) \cong H^0(S, \mathcal{K}_{S}\otimes \mathcal L^*)=\{0\}$ is equivalent to $\mathcal{L} \neq \mathcal{K}_{S}$, where $\mathcal{K}_{S}$ denotes the canonical bundle of $S$.  Let $g$ be a Gauduchon metric whose fundamental $2$-form $F$ satisfies $d_{\alpha}d^c_{\alpha} F=0$ for a $1$-form $\alpha$ with $a=[\alpha] \neq 0$,  given by  Theorem~\ref{thm:conformal}. The corresponding line bundle $\mathcal{L}=\mathcal{L}_a$ has negative degree with respect to $g$ by \eqref{e:degree} whereas it is shown in \cite[Remark~4.2]{Teleman} that the degree of $\mathcal{K}_{S}$ is positive with respect to any Gauduchon metric. Thus, $\mathcal{L} \neq \mathcal{K}_S$, showing that $H^2(S, \mathcal{L})=\{0\}$,  and  therefore $S$ admits a locally conformally symplectic $2$-form taming $J$ with Lee form $\alpha$ by Proposition~\ref{from SKT to symplectic}.

If $S$ is a Hopf surface, Lemma~\ref{l:vanishing2} generally fails, so we cannot directly use Theorem~\ref{thm:conformal} in this case. However, according to  \cite{B,GO}, $S$  admits a lcK metric. \hfill $\Box$

\section{Cohomological invariants of a non-K\"ahler complex manifold}\label{s:cohomological}
An important and well-studied invariant associated to a compact complex manifold  $X=(M,J)$ which admits K\"ahler metrics is its {\it K\"ahler cone}, $\mathcal{K}(X)$, defined to be the subset of classes $\Omega \in H_{dR}^{1,1}(X, \mathbb{R})$ for which there exists a K\"ahler metric on $X$ whose fundamental $2$-form belongs to $\Omega$. A characterization  of $\mathcal{K}(X)$ in terms of the intersection form of the cohomology ring of $X$, the Hodge structure and homology of analytic cycles has been obtained by N. Buchdahl~\cite{buchdahl, buchdahl2} and Lamari~\cite{lamari, lamari2} when  $X$ is a compact complex surface,  and Demailly--Paun~\cite{DP} in general. These results imply in particular that the classes in $H^{1,1}_{dR}(X)$  which contain symplectic forms taming $J$ coincide with  $\mathcal{K}(X)$.

\vspace{0.2cm}
In order to introduce similarly designed cohomological  invariants in the non-K\"ahler lcK case, one can consider  the sets (see ~\cite{OV, B1}):
\begin{defn}  Let $X=(M,J)$ be a  compact complex manifold.  We introduce the following subsets of $H^1_{dR}(M)$:
\begin{itemize}
\item The subset of classes of Lee forms of lcK metrics:
$${\cal C}(X) =\left\{[\a]\mid  \ \exists \ F \in\Omega^{1,1}(X),  F>0, d_{\a} F=0 \right\}.$$
\item The subset of classes of Lee forms of locally conformally symplectic forms taming $J$:
$${\cal T}(X)=\left\{[\a]  \mid \exists \  \o\in \Omega^2(X), \o^{1,1}>0, d_\a \o =0\right\}.$$
\item The subset of classes of the harmonic parts of the Lee forms of  Gauduchon metrics
$${\cal G}(X)=\left\{[\t_h^g]\mid \exists \ F \in \Omega^{1,1}(X), F>0, dd^c F^{m-1}=0, d_{\t} F^{m-1}=0 \right\},$$
where $\t_h^g$ denotes the harmonic part of $\t$ with respect to the riemannian metric $g(\cdot, \cdot) = F(\cdot, J\cdot)$.
\end{itemize}
\end{defn}

\begin{rem}\label{r:goto}
It follows from the definition that
$${\cal C}(X) \subseteq \cal{T}(X),  \ \  {\cal C}(X) \subseteq {\cal G}(X),$$
${\cal G}(X)$ is connected while $\mathcal{T}(X)$ is invariant under small deformations of $X$. Using similar techniques as in \cite{G}, one can show that $a\in \mathcal{T}(X)$ is an interior point, provided that $H^3_{d_{L^*}}(M, L^*) =\{0\}$, where $L=L_{a}$ is the real flat line bundle determined by $a$.
\end{rem}

We now consider the blow-up $\hat X$ of $X$ at a point $x$ and denote $B_x \hat X \to X$ the blow-down map which is a biholomorphism between $\hat X \setminus E \to X \setminus \{x \}$, where $E\cong \mathbb{C}P^{m-1}$ is the exceptional divisor. The following is well-known:

\begin{lem}\label{l:blow-up} Let $B_x : \hat X \to X$ be the blow-down map which contracts a divisor $E\cong {\mathbb C}P^{m-1} \subset \hat X$ with normal bundle $N_E \cong \mathcal{O}(-1)$ to a point $x\in X$. For any $a \in H^1_{dR}(X)$ with generator $\a \in a$, denote by $\hat \a =B_x^*(\a)$. Then $B_x^*: H^{k}_{\a}(X)\to H^k_{\hat \a}(\hat X)$ is surjective for any positive $k\neq 2(m-1)$, and is injective for any positive $k\neq 2m-1$.
\end{lem}
\begin{proof} As $H^{k}_{\a}(X)$ does not depend on the choice of $\a \in a$ (see  the discussion in Sec.~\ref{s:lichnerowicz-novikov}), we can choose $\a$ such that it identically vanishes  on a open ball $U$ centred at $x$. It follows that $\hat \a = B_x^*(\a)$ vanishes on $\hat U = B_x^{-1}(U)$. 

 We shall first prove that $B_x^*: H^k_{\a}(X) \to H^k_{\hat \a}(\hat X)$ is surjective. With our choice for $\a$, any $d_{\hat \a}$-closed $k$-form $\hat \varphi$ on $\hat X$  is closed over $\hat U$.  As $H^{k}_{dR}(\hat U) \cong  H^k_{dR}(\mathbb{C} P^{m-1})=\{0\}$ for $k\neq 2(m-1)$, we can  write ${\hat \varphi}_{\vert \hat U} = d({\hat \xi}_{\vert \hat U})$. Multiplying ${\hat \xi}_{\vert \hat U}$ by the pull-back via $B_x$ of a bump function centred at $x$ and support in $U$, we can assume $\hat \xi$ is globally defined on $\hat X$ and $\hat \phi = \hat \varphi - d_{\hat \a} \hat \xi$ is another form representing $[\hat \varphi] \in H^k_{\hat \a}(\hat X)$ which vanishes identically on a tubular neighbourhood of $E$.  Then, the diffeomorphism $(B_x^{-1}) : \hat X \setminus E \to X\setminus \{x\}$ allows to define a smooth $k$-form $\phi = (B_x^{-1})^* (\hat \phi)$ on $X$ with $d_\a \phi =0$ and $B_x^*(\phi)= \hat \phi$. 
 
 We now prove that $B_x^*: H^k_{\a}(X) \to H^k_{\hat \a}(\hat X)$ is injective. Suppose $\varphi$ is a $d_{\a}$-closed $k$-form on $X$,  such that $\hat \varphi =B_x^*(\varphi) = d_{\hat \a} {\hat \xi}$. As $H^k_{dR}(U)=\{0\}$, we can modify $\varphi$ with a $d_\a$-exact form (as we did above with $\hat \varphi$) and assume without loss that $\varphi_{\vert U} \equiv 0$. It follows that the $(k-1)$-form $\hat \xi$ satisfies $d{\hat \xi}_{\vert \hat U} \equiv 0$. If $k=1$, $\hat \xi$ is a smooth function on $\hat X$ which is constant on $\hat U$ and, therefore, is the pull back to $\hat X$ of a smooth function $\xi$ on $X$ (which is constant on $U$). It follows that $\varphi = d_{\a} \xi$. If $k>1$,  $H^{k-1}_{dR}(\hat U)\cong H^{k-1}({\mathbb C}P^{m-1})=\{0\}$, so we conclude that ${\hat \xi}_{\vert \hat U}$ is exact, i.e. ${\hat \xi}_{\hat U} = d{\hat \eta}_{\hat U}$. Multiplying $\hat \eta$ by a bump-function, we obtain that $\hat \xi - d_{\hat \a} \hat \eta$ is identically zero in a neighbourhood of $E$, so that it descends to $X$ to define a $(k-1)$-form $\xi$ with $d_\a \xi = \varphi$. \end{proof}

We recall the following result established in \cite{tricerri,Vul}. The argument in \cite{Vul} applies without change to the case of locally conformally symplectic structures taming $J$:
\begin{prop}\label{p:blow-up} Let $\hat X$ be the blow-up of $X$ at a point $x$. Then
\begin{enumerate}
\item[\rm (a)] if $a \in \mathcal{T}(X)$, then $\hat a = B_x^*(a) \in \mathcal{T}(\hat X)$;
\item[\rm (b)]  $a \in \mathcal{C}(X)$ if and only if $\hat a= B_x^*(a) \in \mathcal{C}(\hat X)$.
\end{enumerate}
\end{prop}

When $X=S$ is a compact complex surface, the following result is well-known:
\begin{prop}\label{non-kahler} On a compact complex surface $S=(M,J)$ the following conditions are equivalent:
\begin{enumerate}
\item[\rm (i)] $0 \in \mathcal{C}(S)$, i.e. $S$ is K\"ahler;
\item[\rm (ii)] $0 \in \mathcal{T}(S)$;
\item[\rm (iii)] $0 \in \mathcal{G}(S)$;
\item[\rm (iv)] $\mathcal{C}(S) = \mathcal{T}(S) = \mathcal{G}(S)=\{ 0 \}$. 
\end{enumerate}
\end{prop}
\begin{proof} `${\rm (i)} \Rightarrow {\rm (ii)}$' and `${\rm (iv)} \Rightarrow {\rm (i)}$' are obvious.  In order to prove `${\rm (ii)} \Rightarrow {\rm (iii)}$', let $\omega$ be a symplectic $2$-form on $M$ which tames $J$. Letting $F=\omega^{1,1}$ be the positive definite $(1,1)$-part, it defines a hermitian metric $g(\cdot, \cdot)= F(\cdot, J\cdot)$ which is Gauduchon (see e.g. Lemma~\ref{from symplectic to SKT}  with $\alpha=0$). Furthermore, $\omega$ is a closed self-dual $2$-form with respect to $g$, and is therefore co-closed. Writing $\omega = F + \Psi$, where $\Psi$ is of type $(2,0)+(0,2)$, we get 
$$0=\delta \omega =\delta F + \delta \Psi = -J\theta^g - J (\delta {\mathcal J}\Psi), $$
where ${\mathcal J}$ stands for the natural action of the almost-complex structure $J$ on the bundle of real $2$-forms of type $(2,0)+(0,2)$ by $\mathcal{J} \Psi (X, Y): = - \Psi(JX, Y)$. This shows that the Lee form $\theta$ of $g$ is co-exact and, therefore,  $\theta_h^g=0$.

It remains to establish `${\rm (iii)} \Rightarrow {\rm (iv)}$'. To this end, by \cite[Prop.~1]{AG}, $b_1(M)$ is even and $\mathcal{G}(S)=\{0\}$. It follows that $\mathcal{C}(S)=\{0\}$ as $\mathcal{C}(S) \subseteq \mathcal{G}(S)$ and $\mathcal{C}(S)\neq \emptyset$ by the characterization of K\"ahler surfaces~\cite{siu, todorov, buchdahl, lamari}.  

In order to prove $\mathcal{T}(S)=\{0\}$, we use Lemma~\ref{from symplectic to SKT}, Proposition~\ref{p:twisted-gauduchon} and Lemma~\ref{l:integration-by-parts}: as $\theta_h^g=0$ in our case, one gets 
$0 \le -\int_M ||a_h^g||^2_g v_g,$
showing $a_h^g=0$. \end{proof}

Recall the following definition from \cite{OV1}
\begin{defn}\label{d:potential} A {\it lcK metric with potential} $g$ on $X=(M,J)$ is a lcK metric such that the pull-back $\tilde F$ of  its fundamental $2$-form $F$  to the universal covering space $\tilde X$ of $X$ is of the form $\tilde F = \frac{dd^c \tilde f}{\tilde f}$, where $\tilde f>0$ is a positive plurisubharmonic function on $\tilde X$ which satisfies $\gamma^* \tilde f = e^{c_{\gamma}} \tilde f$ ($c_{\gamma} \in \mathbb{R}$)  for any deck-transform $\gamma \in \pi_1(X)$.
\end{defn}
Examples of lcK metrics with potential include the {\it Vaisman} lcK metrics (i.e. lcK metrics for which the Lee form $\theta$ is parallel) or more generally, {\it pluricanonical} lcK metrics,  introduced and studied by G. Kokarev in \cite{kokarev},  for which the covariant derivative $D\theta$ of the Lee form is of type $(2,0)+ (0,2)$ with respect to $J$, see \cite[Claim 3.3]{OV1} and \cite{OV4}.~\footnote{In \cite{OV1}, the authors claim that a lcK metric admits a potential if and only if it is  pluricanonical, but  a proof is given only in one direction; see \cite{OV4} for the precise link between the two notions.}
The following observation is taken from \cite{OV}.
\begin{lem}\label{lcK-potential} Let $X=(M,J)$ be a compact complex manifold endowed with a  lcK metric with potential,  $g$, and Lee form $\theta$. Then for any $t\ge 1$, $t\theta$ is the Lee form of a lcK metric with potential on $X$. If, furthermore, $g$ is a pluricanonical lcK metric, then $t\theta$ is the Lee form of a pluricanonical lcK metric for any $t>0$. 
\end{lem}
\begin{proof} Writing $\tilde F=\frac{dd^c \tilde f}{\tilde f}$  on $\tilde X$, the pull-back of the Lee form is $\tilde \theta = -\frac{d\tilde f}{\tilde f}$. For any $t\ge 1$ put  
$$\tilde F_t := \frac{dd^c {\tilde f}^t}{{\tilde f}^t}= t \frac{dd^c \tilde f}{\tilde f} + t(t-1)\frac{d\tilde f \wedge d^c \tilde f}{{\tilde f}^2}.$$ 
For $t\ge 1$,  $\tilde F_t$  defines a positive definite $(1,1)$-form satisfying  $d\tilde F_t = \tilde \theta_t \wedge \tilde F_t$ with $\tilde \theta _t = -t \frac{d\tilde f}{\tilde f}= t\tilde \theta$. As $\tilde F_t$ is invariant under any deck transformation, it defines a lcK metric with potential, $g_t$,  on $X$,  whose Lee form is $\theta_t = t \theta$.

It is not hard to see that the pluricanonical condition $(D^g_X\theta)(Y) = -(D^g_{JX} \theta)(JY)$ is equivalent to 
\begin{equation}\label{pluricanonical}
dJ\theta = - |\theta|_g^2 F + \theta \wedge J\theta.
\end{equation}
This is essentially the formula appearing in  \cite[p.~724]{OV1}, by noting that there is a sign error in \cite{OV1} in deriving the formula for $d (I \theta)$ from the previous one,  and an omission  of  a factor $|\theta|_g^2$ before $g$  in the formula expressing $\nabla \theta - D\theta$; the precise statement appears in \cite{OV4}. For convenience of the reader we supply  here a brief argument for \eqref{pluricanonical}: it is well-known (see e.g. \cite[II, Prop.~4.2]{KN}) that when $J$ is integrable
$$(D_X F)(Y, Z)= -\frac{1}{2} \Big(dF(X,JY, JZ) - dF(X,Y,Z)\Big).$$
The lcK condition implies $dF = \theta \wedge F$, which allows to re-write the above equality as
\begin{equation}\label{basic}
D_X F = \frac{1}{2}\Big(X^\flat \wedge J\theta + JX^{\flat} \wedge \theta\Big),
\end{equation}
where $X^{\flat}$ denotes the $g$-dual $1$-form to $X$. Thus, using \eqref{basic}, 
\begin{equation*}
\begin{split}
(D_X J\theta)(Y)  & = -\theta\big((D_X J)(Y)\big)  -(D_X\theta)(JY) \\
                          &=  \frac{1}{2}\Big(\theta(X) J\theta (Y)- J\theta(X) \theta(Y) -|\theta|^2_g F(X,Y)\Big) - (D_X \theta)(JY).
                          \end{split}
\end{equation*}
It then follows
$$d J\theta = -|\theta|_g^2 F +\theta \wedge J\theta + (J D\theta)^{\rm anti},$$
where $(J D\theta)^{\rm anti}$ denotes the anti-symmetrization of the $(2,0)$-tensor $(JD\theta)(X,Y):= -D\theta(X, JY)$; as $d\theta=0$ (i.e. $D\theta$ is symmetric), $JD\theta^{\rm anti}(X,Y) =-(D\theta)^{1,1}(X, JY)$ where $D\theta^{1,1}(X,Y):= (D_X \theta)(Y) + (D_{JX} \theta)(JY)$. Thus, for a lcK metric $g$, the equation \eqref{pluricanonical} is equivalent to $D\theta^{1,1}=0$, i.e. to $g$ being pluricanonical.

Differentiating \eqref{pluricanonical} one more time yields 
$$d|\theta|_g^2 \wedge F =0,$$
showing that  $|\theta|_g^2$ is a constant. Thus
$$F_t := F + \frac{t}{|\theta|_g^2}\theta\wedge J\theta, \ t >-1$$
defines a family of positive definite $(1,1)$-forms with $dF_t= (1+t)\theta \wedge F_t$. Clearly, $F_t$ give rise to a family of lcK metrics which verify the pluricanonical condition \eqref{pluricanonical}. \end{proof}
\begin{rem} If $g$ is a pluricanonical lcK (non-K\"ahler) metric on $X=(M,J)$ with Lee form $\theta$, writing the pull-back metric as $\tilde g= e^{\tilde \varphi} \tilde g_K$ where $\tilde g_K$ is a K\"ahler metric on the universal cover $\tilde X$ conformal to $\tilde g$ (or, equivalently,  writing the pull-back $\tilde \theta$ of the Lee form as $\tilde \theta =d\tilde \varphi $) gives rise to a potential function $\tilde f := e^{-\tilde \varphi}$ for the fundamental $2$-form $\tilde F$ of $\tilde g$, i.e. $\tilde F = \frac{dd^c \tilde f}{\tilde f}$ by \eqref{pluricanonical}. \end{rem}

\section{Compact complex surfaces with $b_1(S)=1$}\label{s:surfaces}

We start with the following easy consequence of Proposition~\ref{non-kahler}.
\begin{lem}\label{l:sign} Let $S=(M,J)$ be a compact complex surface with $b_1(M)=1$.   For any $a\in H^1_{dR}(M)$,  denote by $\cal L_a$ the corresponding flat holomorphic line bundle defined via \eqref{H1} and Convention~\ref{convention}. Then,  the sign of ${\rm deg}_g(\cal L_a)$ does not depend on the choice of a Gauduchon metric on $S$ and  is zero if and only if $a=0$.
\end{lem}
\begin{proof} Let $g$ be a Gauduchon metric on $S$. By \eqref{e:degree},  ${\rm deg}_g(\cal L_a)=0$ for any Gauduchon metric if $a=0$. On the other hand, if $a\neq 0$, $\theta_h^g = \mu_a (g) a_h^g$ for a constant  $\mu_a(g)$,  so that by \eqref{e:degree} again, the sign of ${\rm deg}_g (\cal L_a)$ is equal to the sign of $-\mu_a(g)$.  We know by  Proposition~\ref{non-kahler} that $\mu_a(g)\neq 0$. As the space of Gauduchon metrics is convex (and therefore connected), it follows that the sign of $\mu_a(g)$ is non-zero and is independent of $g$. \end{proof}
The above Lemma suggests for the following
\begin{defn}\label{d:order} Let $S=(M,J)$ be a compact complex surface with $b_1(M)=1$. For $a, b \in H^1_{dR}(M)$ we will say that $a>b$   if 
$${\rm deg}_g(\cal L_a \otimes {\cal L}_b^*)= {\rm deg}_g (\cal L_{(a-b)}) = {\rm deg}_g(\cal L_a) - {\rm deg}_g(\cal L_b) >0,$$
for some (and hence any) Gauduchon metric $g$.
\end{defn}
We will thus identify the ordered set $H^1_{dR}(S) \cong (\mathbb{R}, >)$.
\begin{prop}\label{p:trivial} Let $S=(M,J)$ be a compact complex surface with $b_1(M)=1$. Then
$\mathcal{G}(S) \subseteq (-\infty, 0)$ and $\mathcal{T}(S) \subseteq (-\infty, 0).$  Furthermore, for each $a\in \mathcal{T}(S)$, there exists a class $b\le a$ which belongs to $\mathcal{G}(S)$.
\end{prop}
\begin{proof}
Let $g$ be  Gauduchon metric  on $S$ and $\theta_h^g$ the  harmonic part of the corresponding Lee form $\theta^g$. Applying \eqref{e:degree} for $a=[\theta_h^g]$  (and using Proposition~\ref{non-kahler}) yields
${\rm deg}_g (\cal L_a)= -\frac{1}{2\pi} \int_M ||\theta_h^g||_g^2v_g <0,$ thus showing the first inclusion. The second inclusion follows from Proposition~\ref{non-kahler},  and Lemmas~\ref{from symplectic to SKT} \& \ref{vanish}. The inequality $\int_M g(\theta^g_h-a_h^g,a_h^g)v_g =\int_M \big(\frac{1}{\psi}||d\psi||_g^2\big) v_g \ge 0$ (see Lemma~\ref{l:integration-by-parts})  and the fact that $a$ and $b=[\theta_h^g]$ are both negative (so that $\theta_h^g = \mu a_h^g$ with $\mu>0$) show ${\rm deg}_g({\cal L}_{a-b})\ge 0$, i.e. the de Rham class $b \le a$. \end{proof}

\begin{prop}\label{p:non-trivial} Suppose $S=(M,J)$ is a compact complex surface with $b_1(M)=1$. If $b\in \mathcal{C}(S)$, $c \in \mathcal{G}(S)$ with $b\le c$, then any $a\in [b, c]\subset H^1_{dR}(S)$,  such that the corresponding flat line bundle $\mathcal{L}_a$ satisfies $H^2(S, \mathcal{L}_a)=\{0\}$,   belongs to $\mathcal{T}(S)$.
\end{prop}
\begin{proof} Let $g$ be a lcK metric on $S$ whose closed Lee form $\theta^g$ belongs to $b$, and suppose (without loss) that $g$ is a Gauduchon metric: thus,  the Lee form  of $g$  is harmonic, i.e. $\theta^g=b_h^g$. For any $a \in H^1_{dR}(M)$, let $a_h^g$ denote the harmonic representative of $a$ with respect to $g$: thus, writing $a_h^g =  \mu b_h^g$ for a constant $\mu$, we have  $a\ge b$ iff $${\rm deg}_g(\cal L_{b-a})= -\frac{1}{2\pi} \int_M g(b_h^g, b_h^g-a_h^g)v_g= \frac{(\mu-1)}{2\pi} \int_M ||b_h^g||^2_g v_g \le0,$$ i.e. iff $\mu \in [0 ,1]$. Thus,  for any $a\in [b, c]$, 
\begin{equation}\label{lck-L}
\mathbb{L}_{g,a} (\psi) = \Delta_g \psi  + (2\mu-1)g(b_h^g, d\psi) + (1-\mu)\mu ||b_h^g||^2_g \psi.
\end{equation}
We claim that $\lambda_a(g) \ge 0$.   Indeed,  $$\lambda_{a}(g) \ge \min_{M} \mathbb{L}_{g,a}(1)= (1-\mu)\mu \min_{M} ||b_h^g||^2_g \ge 0.$$

Similarly,  let  $\tilde g$ be a Gauduchon metric   on $S$ for which  the de Rham class $c=[{\tilde \theta}_h^{\tilde g}]$ of the harmonic part  of its the Lee form  satisfies $c \ge b$. For any $a\le c$, we have ${\tilde \theta}_h^{\tilde g} = \nu  a_h^{\tilde g}$ for a real constant $\nu \in [0,1]$. It then follows that
$$\int_M \tilde{g}({\tilde \theta}^{\tilde g} - a_h^{\tilde g}, a_h^{\tilde g}) v_{\tilde g}=(\nu-1)\int_M||a_h^{\tilde g}||_{\tilde g}^2v_{\tilde g}  \le 0.$$
As we saw in the proof of Theorem~\ref{thm:conformal}, this implies $\lambda_{a}(\tilde g)\le 0$. 

Considering the linear path $g_t=(1-t)g + t\tilde g, \ t\in[0,1]$ of Gauduchon metrics and using the continuity of $\lambda _{a}(g_t)$ (see Proposition~\ref{p:variational}), one concludes that there exists a Gauduchon metric  $g'$ on $S$ with $\lambda_{a}(g')=0$. Our claim then follows from Proposition~\ref{from SKT to symplectic}. \end{proof}

\section{Examples}\label{s:examples}
In this section we illustrate the previous discussion on the known compact complex surfaces $S$ in the Kodaira class ${\rm VII}$~\cite{bpv}. Note that in this case $b_1(S)=1$.

\subsection{Hopf surfaces} \label{s:hopf} These are, by definition, compact complex
surfaces with universal covering space $\mathbb{C}^2 \setminus \{(0, 0)\}$. It is shown by Kodaira~\cite{kodaira}
that the fundamental group $\Gamma$  of such a surface is a finite extension of the infinite
cyclic group ${\mathbb Z}$. The list of concrete realizations of $\Gamma$ as a group of automorphisms of
$\mathbb{C}^2$ can be found in \cite{kato-hopf},  and we summarize  this classification in the
following rough form:  $\Gamma = H \ltimes \langle \gamma_0 \rangle$, where $\langle \gamma_0\rangle$ denotes  the infinite cyclic group generated by the contraction $\gamma_0$ of $\mathbb{C}^2$
$$\gamma_0(z_1,z_2) = (\alpha z_1 + \lambda z_2^m, \beta z_2),$$
where  the complex numbers  $\alpha, \beta, \lambda$ satisfy $0<|\alpha| \le |\beta| < 1$, $\lambda(\alpha-\beta^m)=0$ for an integer $m \in \mathbb{N}^*$. Furthermore, it follows by the classification in \cite{kato-hopf} that  when $\lambda \neq 0$, $H \subset {\rm U}(1)\times {\rm U}(1)$ is abelian and commutes with $\gamma_0$.
We denote by $S_{\alpha, \beta, \lambda}= \mathbb{C}^2 \ \{(0, 0)\}/\langle \gamma_0 \rangle$ the corresponding {\it primary} Hopf surface and by $S_{\alpha, \beta, \lambda; H}$ the further quotient by $H$, called {\it secondary} Hopf surface. 

Hopf surfaces with $\lambda =0$ are called {\it diagonal} (or, confusingly, of class 1). Belgun has shown~ \cite[Thm.~1]{B} that  any such surface admits a Vaisman lcK metric. By Lemma~\ref{lcK-potential} and Proposition~\ref{p:trivial}, in this case we have $\mathcal{T}(S)= \mathcal{C}(S) = (-\infty, 0)$.

Hopf surfaces with $\lambda \neq 0$ are called {\it resonant} (also called, even more confusingly,  of class 0).  Let $S_{\beta^m, \beta, \lambda;H}$ be a resonant Hopf surface. The analytic family  $S_{\lambda}:= S_{\beta^m, \beta, \lambda;H}, \ \lambda \in \mathbb{C}$ has as central fibre the  diagonal Hopf surface $S_0= S_{\beta^m, \beta, 0; H}$, while for $\lambda \neq 0$ the surfaces $S_{\lambda}$ are isomorphic (see \cite{GO}). As $S_0$ admits a taming conformally symplectic form  with Lee form in any  $a\in (-\infty, 0)$ so does $S_{\lambda}$, just by continuity using that the taming condition is open. It follows that $\mathcal{T}(S_{\lambda})= (-\infty, 0)$.

Similarly, as $S_0$ admits a Vaisman metric by \cite{B,GO}, it has a Vaisman lcK metric $g_a$ with  fundamental $2$-form $F_a$ and Lee form in $a$ for any $a\in (-\infty, 0)$, by Lemma~\ref{lcK-potential}. (Recall that any  Vaisman lcK metric  is pluricanonical.) As the corresponding $H$ commutes with $\gamma_0 $ in this case, the potential $\tilde f$ for $F_{a}$ can be chosen to be $H$-invariant (by averaging over $H$). By the argument in \cite{GO} , the $\tilde f$ can be deformed to define an $H$-invariant potential for a lcK metric on $S_{\lambda}$,  for $\lambda$ small enough, with the same constant $c_{\gamma_0}$ (see Definition ~\ref{d:potential}). Using the isomorphism $\mathbb{C}^{*} =H^1(S, \mathbb{C}^*) \cong {\rm Pic}_0(S)$ established for class VII surfaces in \cite[I, p.~756]{kodaira},  this shows that the induced lcK metrics on $S_{\lambda}$ will have Lee forms in the same de Rham class $a$. We thus see that $\mathcal{C}(S_{\lambda}) =(-\infty, 0)$ too.

Using Proposition~\ref{p:blow-up} (and Lemma~\ref{l:blow-up} with $a=0$ and $k=1$), we conclude
\begin{prop}\label{p:hopf} For any  compact complex surface $S$ whose minimal model is a Hopf surface, $\mathcal{T}(S)=\mathcal{C}(S)=(-\infty, 0)$.
\end{prop}

\subsection{Inoue surfaces with $b_2=0$}~\label{s:inoue}  The  Inoue surfaces $S_0$ with second Betti number equal to zero are classified by Inoue \cite{inoue} into three types, $S_M, S^-_{N;p,q,r}$ and $S^+_{N;p,q,r; u}$,  where  the parameters $M, N$ are matrices with integer coefficients, $p,q,r$ are integers and $u$ is a complex number.  Any such surface is the quotient of $\mathbb{C} \times {\mathbb H}$ (where ${\mathbb H}$ denotes the upper half-plane in $\mathbb C$) under a discrete subgroup $\Gamma$  of the group $A(2, \mathbb{C})$ of affine trasformations of $\mathbb{C}^2$, leaving $\mathbb{C} \times {\mathbb H}$ invariant. The specific description of $\Gamma$ in each case is given in \cite{inoue}, but we shall not make use of this.  The  relevant information for the discussion below is the fact, shown by Tricerri in \cite{tricerri},   that all Inoue surfaces admit lcK metrics, except the surfaces of the type $S_u:=S^+_{N;p,q,r; u}$ for which the complex parameter $u$ is not real. In the later case, Belgun~\cite{B} shows that there are no lcK metrics at all.  Nevertheless,  we proved in Theorem~\ref{thm:lcs} that  $S_u$ always admits timing locally conformally symplectic structures.

Let us now consider in a little more detail the analytic family $S_u, u \in \mathbb{C}\setminus \mathbb{R}$ of Inoue surfaces of the third type.  We shall show, by using an argument from \cite{B},  that in this case $\mathcal{T}(S_u)$ is a single point. 

It is known \cite[p.~35]{B}, \cite[Thm.~1]{H}, \cite{wall} that  $S_u = (\mathbb{C}\times \mathbb{H})/\Gamma_u$ where $\Gamma_u$ is a lattice in the solvable Lie group
$${\rm Sol}_1'^4 =\left\{\left( \begin{array}{ccc} 1 & a &  b + i \log \gamma \\ 0 & \gamma & c \\ 0 & 0 & 1 \end{array} \right ), \ \gamma >0, a,b,c \in \mathbb R\right\} \subset {\rm GL}(3, \mathbb{C}).$$
The group ${\rm Sol}_1'^4$  acts itself simply-transitively (and holomorphically) on $\mathbb{C}\times \mathbb{H}$. This allows to identify  $S_u$
with  the quotient  ${\rm Sol}_1'^4$,  endowed with a (fixed)  left-invariant integrable almost complex structure $J$,  by  the left action of the lattice $\Gamma_u$. In explicit terms,  let
$$Y=\left( \begin{array}{ccc} 0 & 1 &  0 \\ 0 & 0 & 0 \\ 0 & 0 & 0 \end{array} \right )  Z=\left( \begin{array}{ccc} 0 & 0 &  1\\ 0 & 0 & 0 \\ 0 & 0 & 0 \end{array} \right ) T=\left( \begin{array}{ccc} 0 & 0 &  i \\ 0 & 1 & 0 \\ 0 & 0 & 0 \end{array} \right ) U=\left( \begin{array}{ccc} 0 & 0 &  0 \\ 0 & 0 & 1 \\ 0 & 0 & 0 \end{array} \right )$$
be the generators of the Lie algebra of ${\rm Sol}_1'^4$ with
$$Z \ {\rm central}, [Y, T]=Y, [T, U]=U, [Y, U]=Z$$
and denote with the same letters the induced left-invariant vector fields  on ${\rm Sol}_1'^4$. Then the left-invariant complex structure $J$ on ${\rm Sol}_1'^4$ is given by \cite[(35)]{B}
$$J Y = -Z, JZ= Y, JT=-U-Z,  JU= T-Y.$$
Furthermore, ${\rm Sol}_1'^4$ admits an ad-invariant  $1$-form $\alpha_0$, defined by $\alpha_0(T)=1, \alpha_0(U)=\alpha_0(Y)=\alpha_0(Z)=0$, which descends to define a closed but not exact $1$-form (still denoted $\alpha_0$) on  $S_u$. As $b_1(S_{u})=1$, by Remark~\ref{r:conformal}  we can assume without loss that $S_{u}$ admits a locally conformally symplectic $2$-form $F$ which tames $J$,  and whose Lee form equals $k\alpha_0$ for a non-zero real constant $k$. As ${\rm Sol}_1'^4$ has a bi-invariant volume form $v$~\cite[Lemma~4]{G} (which defines a volume form on the quotient $S_u$, still denoted by $v$), for any left-invariant vector fields $U$ and $V$ on  ${\rm Sol}_1'^4$ (which define vector fields on the quotient $S_u$, still denoted by $U$ and $V$), one can consider the average of $F$ over $S_{u}$:
$$F_0(U, V) := \int_{S_u} F(U, V) v.$$
It  can be shown,  as in the proof of \cite[Thm.~7]{B},  that $F_0$ defines  a {\it left-invariant} $2$-form  on $({\rm Sol}_1'^4,J)$ which tames $J$ and satisfies $dF_0 = k\alpha_0 \wedge F_0$. Evaluating the later equality over $Y,Z,T$ yields (see also \cite[(36)]{B}) $F_0(Y,Z)=kF(Y,Z)$; as $Y=JZ$ and   the $(1,1)$-part of $F_0$ is positive definite,  it follows that $k=1$.  This shows that $\mathcal{T}(S_u)=\{a_0\}$ with $a_0=[\alpha_0]$.  It is easy to check (see  \cite[Sect.~7.1]{G}), using the explicit description of $S_{u}$ of \cite{inoue} and the isomorphism   $\mathbb{C}^{*} =H^1(S, \mathbb{C}^*) \cong {\rm Pic}_0(S)$ of \cite[I, p.~756]{kodaira},  that the corresponding holomorphic line bundle $\mathcal{L}_{a_0}$ is isomorphic to the anticanonical line bundle $\mathcal{K}^*_{S_u}$.

Noting finally that the existence of a locally conformally symplectic structure taming $J$ is an open condition (with respect to $J$) and that the Inoue surfaces of the type $S^-_{N, p, q, r, s}$ are quotients of Inoue surfaces of the type $S^+_{N, p,q,r; 0}$ by an involution~\cite{inoue}, we obtain the following
\begin{prop}\label{p:inoue} Let $S$ be an Inoue surface with $b_2(S)=0$ in one of the types $S^+_{N, p,q,r; u}$ or $S^-_{N; p,q,r}$. Then $\mathcal{C}(S)$ and $\mathcal{T}(S)$ are given as follows:
\begin{enumerate}
\item[$\bullet$] $\mathcal{C}(S)=\mathcal{T}(S)=\{ a_0\}$ iff $S$ is of the type $S^-_{N; p,q,r}$ or $S^+_{N, p,q,r; u}$ with $u\in {\mathbb R}$;
\item[$\bullet$] $\mathcal{C}(S)=\emptyset, \ \mathcal{T}(S)=\{a_0\}$ iff $S$ is of the type $S^+_{N, p,q,r; u}$ with $u\in \mathbb{C}\setminus {\mathbb R}$,
\end{enumerate}
where $\mathcal{L}_{a_0} = \mathcal{K}^*_{S_{u}}$.
\end{prop}

\subsection{Kato surfaces}\label{s:kato} These are minimal complex surfaces in the Kodaira class ${\rm VII}$ whose second Betti number is  strictly greater than $0$, and 
which have a {\it global spherical shell} (GSS). Conjecturally, any minimal surface in the class VII should be either a Hopf surface, an Inoue surface,  or a Kato surface. However, this conjecture is still far from being solved. 

For Kato complex surfaces, Brunella~\cite{B1,B2} has shown that $\mathcal{C}(S) \neq \emptyset$ and that $\mathcal{C}(S)$ has $-\infty$ as an accumulation point. Note that any Kato surface  $S$ is diffeomorphic to $(S^1\times S^3)\sharp k \overline{{\mathbb C}P^2}$ (where $k:=b_2(S)$), see e.g.~\cite{Nakamura}.  As $S^1\times S^3$,  with a complex structure of a Hopf surface of class 1,  admits a Vaisman metric,  it follows by \cite{leon,OV2} that $H^3_{d_{L}}(S^1\times S^3, L)=\{0\}$ for any non-trivial flat real line bundle $L$.  By Lemma~\ref{l:blow-up}, $H_{d_{L}}^3(S,L)=\{0\}$ and,  therefore, by  Remark~\ref{r:goto}, $\mathcal{T}(S)$ must be an open subset of $(-\infty, 0)$. Similar conclusion holds true for $\mathcal{C}(S)\subseteq \mathcal{T}(S)$, by  \cite{G} and Lemma~\ref{vanishing}, together with the vanishing of $H^2(S, \mathcal L)$ and $H^0(S, \mathcal{L})$  established in Lemma~\ref{l:vanishing2} and  Lemmas~\ref{from symplectic to SKT} and \ref{vanish},  respectively.

Further progress in this case seems to depend on a better understanding of the subset $\mathcal{G}(S)$, as the following result suggests.
\begin{prop} Let $S$ be a Kato surface. Then,
$$\mathcal{G}(S) \subseteq \mathcal{T}(S).$$
\end{prop}
\begin{proof} Let $c \in \mathcal{G}(S)$. Brunella shows (see \cite[Rem.~9]{B1}) that there exists a lcK metric whose Lee form defines a de Rham class  $b$ with $b<c$. By Proposition~\ref{p:non-trivial} and Lemma~\ref{l:vanishing2}, any $a\in [b,c]$ (in particular $a=c$) belongs to $\mathcal{T}(S)$. \end{proof}

\appendix
\section{A Perron-type Theorem  for second-order strongly elliptic linear operators on  a compact Riemannian manifold}\label{a:PDE}
We review here some spectral properties of the second order strongly elliptic linear operators in the form
$$L(u)= \Delta^g u + g(\alpha, du) + c u,$$
where $g$ is a Riemannian metric on a compact manifold $M$, $\Delta^g= \delta^g d $ is the corresponding Riemannian Laplacian (which we shall consequently denote by $\Delta$), $\alpha$ is a given smooth $1$-form and $c$ a given smooth function on $M$.  Note that $L$ need not to be self-adjoint in general, a case where the spectral theory is well-established.

In a local chart $U$, $L$ takes the form
$$L_U=-\sum_{i,j=1}^n a^{ij}_U(x)\frac{\part^2}{\part x_i\part x_j} + \sum_{k=1}b^k_U(x)\frac{\part}{\part x_k} +c_U(x),$$
where $a_U^{ij}, b_U^k,c_U$ are smooth real functions such that the symmetric matrix $(a_U^{ij})$ is uniformly positive definite on $U$.  In this case, the Perron-type Theorem established in  \cite[p.~360]{evans} and \cite{LeiNi} states that if $\Omega \subset U$ is compactly supported domain with smooth boundary, the operator $L_{U}$ taken on smooth functions on $\bar \Omega$ with zero boundary value has {\it real} eigenvalue $\lambda_0$ (called {\it principal eigenvalue}) such that for any other eigenvalue $\lambda$, ${\rm Re}(\lambda) \ge \lambda_0$; furthermore $\lambda_0$ is simple and the corresponding eigenspace is generated by a nowhere vanishing  smooth function $u_0$ on $\Omega$. 

Recall that for any strongly elliptic linear operator $L : \mathcal{C}^{\infty}(M) \to \mathcal{C}^{\infty}(M)$, the set of (complex) eigenvalues is discrete,  having a limit point only at infinity, see  e.g.~\cite[p.~465]{Besse} or \cite[p.~126]{aubin}. We want to establish  the following adaptation of the Perron-type Theorem  mentioned above  to the case of a compact manifold without boundary (as we failed to find a reference to this result in the literature).

\begin{thm}\label{A:1} Let $L(u)=\Delta u + g(\alpha, du) + c u$ be a linear strongly elliptic linear differential operator of order 2  on a compact Riemannian manifold  $(M,g)$. Then, there exists a real eigenvalue $\l_0$ for $L$ which admits a smooth everywhere positive eigenfunction $u_0$. Furthermore, $\lambda_0$ satisfies the following properties:
\begin{enumerate}
\item[(i)]  $\lambda_0$ is of multiplicity one.
\item[(ii)] If  $\l$ is another eigenvalue, then ${\rm Re}(\l)>\l_0$.
\item[(iii)] $\l_0=\sup_{u\in\cal A}\Big(\inf_{x\in M}\frac{L(u)}{u}\Big)$, where $\cal A=\{u\in\cal C^\infty(M)\mid u>0 \}$.
\end{enumerate}
\end{thm}

\begin{defn} The eigenvalue $\l_0$ is called the {\it principal eigenvalue} of $L$.
\end{defn}

\vspace{0.2cm}
\noindent
{\bf Proof of Theorem~\ref{A:1}.}  The proof will be decided into four steps,  corresponding to the statements in the Theorem~\ref{A:1} as follows.

\vspace{0.2cm}
\noindent
{\bf Step 1.} We shall establish in this Step the existence  of a real eigenvalue $\lambda_0$  corresponding to an everywhere positive smooth eigenfunction $u_0$. To this end, we shall work with the Sobolov spaces $W^2_k(M)$  corresponding to the norm 
$$\|f\|_{k}: =\left[\sum_{0\le j\le k} \int_M|D^jf|^2 v_g\right]^{1/2},$$
where $|D^jf|$ is the point-wise norm of the $j$-th covariant derivative and $v_g$ is the Riemannian volume form.  In what follows, we shall choose $k$ sufficiently large so that we have a  continuous embedding  $W^2_{k}(M)\subset \cal C^2(M)$. Because $M$ is compact,  we also have the continuous embeddings
$$W^2_k(M)\subset W^2_l(M), $$
for any integers $0\le l \le k$.

We shall assume (without loss) in what follows that the smooth function $c\ge 0$ (otherwise we consider the operator $L - ({\rm inf}_M c) {\rm Id}$ instead of $L$). Then, by the maximum principle~\cite[III, Sect.~8, 3.71]{aubin} (which in the sequel we shall always apply to $-L$ and $-u$),  $0$ is not an eigenvalue of $L$, i.e. ${\rm Ker} L =\{0\}$.  It is a standard fact that $L:W^2_{k+2}(M)\to W^2_k(M)$ is then invertible,  and $L^{-1}:W^2_k(M)\to W^2_{k+2}(M)$ is bounded. Indeed, as $L = \Delta + T$ with $\Delta$ being the self-adjoint Riemannian Laplacian and $T$ of order $\le 1$, the composition $T : W^2_{k+2}(M)\to W^2_{k+1}(M)\subset W^2_k(M)$ is a compact operator by the Rellich--Kondrachov theorem  (see \cite[p.~458]{Besse})); it follows that  ${\rm Index}(L)={\rm Index}(\D)=0$ by \cite[Cor.~19.1.8]{hormander} and,  as ${\rm Ker} \ L=\{0\}$,     $L$ is invertible and  bounded (by the standard $L^p$-estimates, see e.g. \cite[p. 463]{Besse}). Using the compactness  of $W^2_{k+1}(M)\subset W^2_k(M)$ again, we conclude that the composition 
$$A : W^2_k(M)\stackrel{L^{-1}}{\longrightarrow}W^2_{k+2}\subset W^2_k(M)$$ is  a compact operator.  
We define the cone
\begin{equation}\label{cone}
C:=\{u\in W^2_k(M)\mid u\ge 0\}.
\end{equation}
By the maximum principle~\cite[III, Sect.~8, 3.71]{aubin}  again, we have that  $A(C)\subset C$ and, if $w\in C$ with $w\not\equiv 0$,  then $Aw >0$. Using standard elliptic regularity  (see e.g. \cite[p. 467]{Besse}) it will be  enough to show that $A$ has  a non-trivial eigenfunction  $w_0\in C$,  corresponding to a real positive eigenvalue $1/\lambda_0$ (we will have then that $u_0:=\lambda_0 A(w_0)$ is a smooth strictly positive eigenfunction of $L$,  
 corresponding to the real positive eigenvalue $\lambda_0$).

Let us fix a function $w\in C$, $w\not\equiv  0$,  and let $v:=Aw$. By continuity, there exists $\mu>0$ such that
$$\mu v\ge w$$  on $M$. We claim that  if the equation
$$u=\lambda A(u+\e w)$$
with  $\e>0$ and $\lambda>0$  have  an everywhere positive solution $u$, then necessarily 
\begin{equation}\label{inequality}
\lambda\le \mu.
\end{equation}
Indeed, 
$$u=\lambda A(u)+\lambda A(\e w)> \lambda A(\e w)=\lambda\e v\ge \Big(\frac{\lambda}{\mu}\Big) \e w$$
hence
$$u\ge \lambda A(u)\ge \lambda A\left(\frac{\lambda\e}{\mu}w\right)=\frac{\lambda^2}{\mu} \e v \ge \Big(\frac{\lambda}{\mu}\Big)^2\e w.$$
By induction,  for $k\ge 1$,
$$u\ge \left(\frac{\lambda}{\mu}\right)^k \e w,$$
which is possible only  if $\lambda\le \mu$.

\vspace{0.2cm}
We are now going to show that  for any $\epsilon >0$, the closed subset of $W^2_k(M)$ 
$$S_\e:=\{ u\in C\mid \exists \ \lambda, 0\le \lambda\le 2\mu, u=\lambda A(u+\e w)\}$$
is unbounded. To this end, we shall use \eqref{inequality} in conjunction with  the well-known Schaefer theorem (see \cite[p.~540]{evans}), which states that 
if $\tilde A : W \to W$ is  a continuous compact (not necessarily linear) mapping of a Banach space $W$ and  $C\subset W$ is a convex subset stable by $\tilde A$, then $\tilde A$ has a fixed point in $C$ provided  that the subset $$\{u\in C\mid \exists \ \l, \ 0\le \l\le 1\ {\rm such\ that\ } u=\l \tilde A (u)\}$$ is bounded. 

In our case, $W= W^2_k(M)$, 
$$\tilde A (u) := 2\mu\bigl(A(u)+\e v\bigr)$$
which is continuous compact  because $A$ is, and $C\subset W^2_k(M)$ is the convex subset introduced in \eqref{cone}. Furthermore, if $S_\e$ were bounded, $\tilde A$ would satisfy  the hypothesis of the Schaefer theorem with respect to $C$, so $\tilde A$ would have a fixed point $u$ in $C$, i.e. $u={\tilde A} u=2\mu A(u+\e w)$,  which contradicts \eqref{inequality}.

Since $S_\e$ is not bounded, there exists $u_\e \in S_\e$  with $\|u_\e\|_k\ge \frac{1}{\e}$. Denote by $\lambda_\e$ ($0\le\lambda_\e\le 2\mu$) the corresponding real number  such that
$$u_\e = \lambda_\e A(u_\e+\e w).$$
For any sequence $\e_m \to 0$ let  $\lambda_m:=\lambda_{\e_m}$ and $u_m:=\frac{u_{\e_m}}{\|u_{\e_m}\|_k}\in C$, so that
\begin{equation}\label{limit}
u_m=\lambda_m A\Big(u_m+\e_m \frac{w}{\| u_{\e_m}\|_k}\Big).
\end{equation}
Since $A$ is compact, we can take a sub-sequence if necessary and suppose that $(\lambda_m)$ and $(u_m)$ are convergent. Let $\lambda_0:=\lim_{m\to \infty} \lambda_m$ and $u_0:=\lim_{m \to \infty} u_m\ge 0$.  Taking limit in \eqref{limit}, we get 
$$u_0=\lambda_0 A(u_0)$$
The condition $\|u_m\|_k=1$ prevents $\lambda_0$ and $u_0$ to vanish. Thus, $u_0= \lambda_0A(u_0)>0$, and applying $L$ to the above equality we obtain $L(u_0)= \lambda_0 u_0$ with $\lambda_0>0$.

\vspace{0.2cm}
\noindent
{\bf Step 2.}  We shall now prove that  the multiplicity of $\l_0$ is one. 

As $L$ is  a real operator, it is enough to consider  a smooth non-identically zero real valued function $u$ on $M$,  such that $Lu=\l_0 u$. Furthermore, replacing $u$ by $-u$ if necessary, we shall assume that $u$ is somewhere positive.  Thus, letting
$$\chi=\sup\{\mu>0\mid u_0-\mu u\ge 0\ {\rm on}\ M\},$$
we then have $\chi>0$, $v:=u_0-\chi u\ge 0$ on $M$,  and hence $Lv=\l_0 v\ge 0$. 

By the Hopf maximum principle~\cite[III, Sect.~8, 3.71]{aubin},  we have that either  $v>0$ on $M$ or $v\equiv 0$. By the definition of $\chi$ we conclude  that $v\equiv 0$.

\vspace{0.2cm}
\noindent
{\bf Step 3.} We now show that for any other (complex) eigenvalue $\lambda$ of $L$, ${\rm Re}(\lambda)>\lambda_0$.

Let  $u\not\equiv  0$ be a complex-valued smooth function on $M$ with  $Lu=\l u$ and let $v:= u/u_0$. We then have
\begin{equation*}
\begin{split}
\lambda u_0 v = L(u_0v)& =  v L(u_0) + u_0 L(v) - c u_0v - 2g(du_0, dv) \\
                                             &= \lambda_0 u_0 v + u_0\Big(\Delta v   + g (\alpha - 2d  \log u_0, dv)\Big)\\
                                             &= \lambda_0 u_0 v + u_0 K(v),
\end{split}
\end{equation*}
where we have set
$$K(v):= \Delta (v) +    g (\alpha - 2d  \log u_0, dv).$$  
Dividing by $u_0$ the last equality, we have
$$K(v) =(\lambda - \lambda_0) v,$$    
and as $K$ is a real  operator
$$K(\bar v) =    (\bar \lambda - \lambda_0) \bar v.$$
It follows
\begin{equation}\label{constant}
\begin{split}
K(|v|^2) = K(v \bar v) & = v K(\bar v) + \bar v K(v) - 2g(dv, d\bar v) \\
                                     &= 2\Big({\rm Re}(\lambda) - \lambda_0\Big)|v|^2 - 2g(dv, d\bar v)\\
                                     & \le    2\Big({\rm Re}(\lambda) - \lambda_0\Big)|v|^2.
                                     \end{split}
\end{equation}                                                    
Suppose for contradiction that  ${\rm Re}(\lambda) - \lambda_0 \le 0$. Then, the Hopf maximum principle~\cite[III, Sect.~8, 3.71]{aubin} implies $|v|=const$, and therefore 
$$0=K(|v|^2)\le 2\Big({\rm Re}(\lambda) - \lambda_0\Big)|v|^2$$
showing that ${\rm Re}(\lambda) =\lambda_0$ (as $v=u/u_0 \not\equiv 0$).  Going back to \eqref{constant} (and using again $K(|v|^2)=0$), we have
\begin{equation*}
\begin{split}
0= & K(|v|^2)= -2g(dv, d\bar v) + 2\Big({\rm Re}(\lambda) - \lambda_0\Big)|v|^2\\
 = & -2||dv||^2_g
 \end{split}
 \end{equation*}
 i.e. $v=u/u_0$ is a constant. As $\lambda \neq \lambda_0$ by assumption, this is a contradiction.
 
\vspace{0.2cm}
\noindent
{\bf Step 4.} We finally have to prove that 
$$\l_0=\sup_{u\in\cal A}\inf_{x\in M}\frac{L(u)}{u}$$
where $\cal A=\{u\in\cal C^\infty(M)\mid u>0 \}$.

To this end, we shall show first that the formal conjugate operator $$L^*(u)=\Delta u -g(\alpha, du) + (\delta^g \alpha + c) u$$ of $L$ (with respect to the global $L_2$ product on $(M,g)$) has the same principal eigenvalue as $L$. Indeed, let $\lambda_0^*$ be the principal eigenvalue of $L^*$ and $u^*_0>0$ an eigenfunction. Then
$$\l_0\int_Mu_0u_0^* v_g =\int_ML(u_0) u_0^* v_g=\int_Mu_0L^*(u_0^*) v_g= \l_0^* \int_Mu_0u_0^* v_g.$$
Since $u_0>0$ and $u_0^*>0$, it follows that $\l_0=\l_0^*$.

Now, let 
$$\mu:= \sup_{u\in\cal A}\inf_{x\in M}\frac{L(u)}{u}. $$
We clearly have $\mu\ge \inf_{x\in M}\frac{L(u_0)}{u_0}=\l_0$. To establish the converse inequality, let $(w_m)$ be a maximizing sequence in $\cal A$ such that 
\begin{equation}\label{maximazing}
\mu-\frac{1}{m}\le \inf_{x\in M}\frac{L(w_m)}{w_m}\le \mu.
\end{equation}
It then follows
$$\l_0\int_M w_m u_0^*\, v_g=\int_M w_m L^*(u_0^*)\, v_g =\int_ML(w_m)u_0^*\,v_g \ge (\mu-\frac{1}{m})\int_Mw_nu_0^*\,v_g,$$
thus showing $\l_0\ge \mu$. \hfill$\Box$

The following result follows from the general Rellich--Kato theory, see e.g. \cite[II, Sect.~1.8; IV, Theorem 3.16,  and VII, Thm.~1.7]{kato} applied to the bounded operator $A$ constructed in the proof of Theorem~\ref{A:1}, by noting that the principal eigenvalue $\lambda_0$ of $L$ is simple and can be separated from the reminder of the spectrum.
\begin{thm}\label{A:2} Let $L(t)$ be an analytic  family of  linear strongly elliptic operators  as in Theorem~\ref{A:1}.  For each $t$,  denote by  $\lambda_0(t)$ the principal eigenvalue  of $L(t)$ with corresponding eigenfunction $u_t>0$,  normalized by $\int_M u_t^2 v_g=1$. Then $\lambda_0(t)$ and $u_t$ vary  analytically with respect to $t$.
\end{thm}

\end{document}